\documentclass[draft,reqno,11pt]{amsart}

 \usepackage{amssymb,amsmath, geometry}
\numberwithin{equation}{section}
\newtheorem{theorem}{Theorem}[section]
\newtheorem{proposition}{Proposition}[section]
\newtheorem{lemma}{Lemma}[section]

\DeclareMathOperator{\Scal}{R}

\DeclareMathOperator{\R}{R}

\DeclareMathOperator{\Weyl}{Weyl}

\DeclareMathOperator{\Span}{Span}

 \newcommand{\<}{\left\langle}
\renewcommand{\>}{\right\rangle}
 \renewcommand{\(}{\left(}
\renewcommand{\)}{\right)}
\renewcommand{\[}{\left[}
\renewcommand{\]}{\right]}
\newcommand{\eps}{\varepsilon}
\begin{document}

\title[Clustering phenomena for linear perturbation of the Yamabe equation]{Clustering phenomena for linear perturbation of the Yamabe equation }

\author{Angela Pistoia}
\address{Angela Pistoia, Dipartimento di Scienze di Base e Applicate per l'Ingegneria, Sapienza Universit\`a di Roma,
via Antonio Scarpa 16, 00161 Roma, Italy}
\email{angela.pistoia@uniroma1.it}

\author{Giusi Vaira}
\address{Giusi Vaira, Dipartimento di Scienze di Base e Applicate per l'Ingegneria, Sapienza Universit\`a di Roma,
via Antonio Scarpa 16, 00161 Roma, Italy}
\email{vaira.giusi@gmail.com}

\maketitle

\centerline{\em This paper is  warmly dedicated to Professor Abbas Bahri  on the occasion of his 60th birthday}
 
\date{}
\begin{abstract}

 Let $(M,g)$ be a non-locally conformally flat  compact Riemannian manifold  with dimension $N\ge7.$  We are interested in finding positive solutions to the  linear perturbation of the Yamabe problem
$$-\mathcal L_g u+\epsilon u=u^{N+2\over N-2}\ \hbox{in}\ (M,g) $$
 where the first eigenvalue of the conformal laplacian  $-\mathcal L_g $ is positive and $\epsilon$ is a small positive parameter.
We prove that  for  any point $\xi_0\in M$ which is non-degenerate and non-vanishing minimum point of the Weyl's tensor and  for any integer $k$ there exists a family of  solutions developing $k$ peaks collapsing at $\xi_0$ as $\epsilon$ goes to zero. In particular, $\xi_0$ is a non-isolated   blow-up point.

\end{abstract}


\section{Introduction}
Let $\(M,g\)$ be a smooth, compact Riemannian manifold of dimension $N\ge3$.
 The  Yamabe problem consists in 
finding   metrics   of constant scalar curvature 
  in the conformal class of $g$. It is equivalent to finding a positive solution to the problem
  \begin{equation}\label{ya}
\mathcal L_g u+\kappa u^{N+2\over N-2}=0\ \hbox{in}\ M,\end{equation}
for some constant $\kappa$.
Here $\mathcal L_gu:= \Delta_g u+{N-2\over 4(N-1)}R_gu$ is the conformal laplacian, $\Delta_g$ is the Laplace-Beltrami operator and   $R_g$ is the scalar curvature of the manifold .

In particular,  if $u $ solves \eqref{ya}, then the scalar curvature of the metric $\tilde g= u^{4\over N-2} g$ is nothing but ${4(N-1)\over N-2}\kappa.$
Yamabe problem has been completely solved by Yamabe \cite{y}, Aubin \cite{a}, Trudinger \cite{t} and Schoen \cite{s1} (see also the proof given by Bahri \cite{ba}).
The solution is unique  in the case of negative scalar curvature and it is unique (up to a constant factor)  in the case of zero scalar curvature. The uniqueness is not true anymore in the case of positive scalar curvature. 
Indeed,   Schoen \cite{s2} and Pollack in \cite{p} exhibit examples where
  a large number of high energy solutions with high Morse
index exist. Thus it is natural to ask if the set of solutions is compact or not as it was raised by Schoen  in \cite{s3}.
It is also useful to point out that in the case of the round sphere $(\mathbb S^{N},g_0)$ the  compactness does not hold (see Obata in \cite{o}). Indeed,
the scalar curvature $R_{g_0}=N(N-1)$ and
the Yamabe problem \eqref{ya} reads as
 $$ -\Delta _{g_0}u+ {N(N-2)\over 4} u=u^{N+2\over N-2} \ \hbox{in}\ (\mathbb S^N,g_0) $$
  which is equivalent   (via the stereographic projection) to the equation in the Euclidean space
  \begin{equation}\label{criti} -\Delta U=U^{N+2\over N-2} \  \hbox{in}\  \mathbb R^N.\end{equation}
It is known that \eqref{criti} has infinitely many solutions, the so called {\em standard bubbles},
  \begin{equation}\label{bubble}U_{\mu,y}(x) = \mu^{-{N-2\over2}} U\({  x-y \over \mu}\) ,\ x,y\in\mathbb R^N,\ \mu>0,\ \hbox{where}\
 U (x):=\displaystyle{ {\alpha_N}{1\over\(1+|x |^2\)^{N-2\over2}}}.\end{equation} 
Here $\alpha_N:={N(N-2)}^{N-2\over4}.$

The compactness turns out to be true when the dimension of the manifold satisfies $3\le N\le 24$ as it was shown by Khuri, Marques and Schoen \cite{kms})
(previous  results were obtained  by Schoen \cite{s4}, Schoen and  Zhang \cite{sz}, Li and  Zhu \cite{lz}, Li and  Zhang \cite{lzha}, Marques \cite{m} and Druet \cite{d1}),
while it is false when $N\ge25$ thanks to the examples built by  Brendle \cite{b} and Brendle and Marques \cite{bm}.
 The proof of compactness strongly relies on proving sharp pointwise estimates at a blow-up point  of the solution.   In particular, when compactness holds every sequence of unbounded solutions to  \eqref{ya} must blow-up  at some points of the manifold which are necessarily  isolated and simple, i.e. 
  around each blow-up point $\xi_0$ the solution can be approximated by a standard bubble (see \eqref{bubble})
 $$u_n(x)\sim \alpha_N{ \mu_n^  {N-2\over 2}\over \(\mu_n^2+\(d_g(x,\xi_n\))^2 \)^{N-2\over 2}} \ \hbox{
 for some $\xi_n\to\xi_0$ and $\mu_n\to0.$}$$
 
 More precisely, 
 let $u_n$ be a sequence of solutions to problem \eqref{ya}. We say that $u_n$  blows-up at a point $\xi_0\in M$ if    there exists $\xi_n\in M$ such that
 $\xi_n\to \xi_0\ \hbox{and}\ u_n(\xi_n)\to+\infty.$   $\xi_0$ is said to be a {\em blow-up point} for $u_n.$ 
 Blow-up points can be classified according to the definitions introduced by  Schoen  in \cite{s3}.
$\xi_0\in M$ is an {\em isolated blow-up point} for $u_n$ if there exists $\xi_n\in M$ such that
  $\xi_n$ is a local maximum of $u_n,$ 
 $\xi_n\to \xi_0,$ 
   $u_n(\xi_n)\to+\infty$ and
   there exist $c>0$ and $R>0$ such that $$0<u_n(x)\le c{1\over d_g(x,\xi_n)^{ {N-2\over2}}}\ \hbox{for any}\ x\in B(\xi_0,R).$$  
 Moreover, $\xi_0\in M$ is an  {\em isolated simple blow-up point} for $u_n$ if  
    the function
    $$\hat u_n(r):=r^{N-2\over 2}{1\over |\partial B(\xi_n,r)|_g}\int\limits_{\partial B(\xi_n,r)}u_n d\sigma_g$$
    has a exactly one critical point in $(0,R).$  
 \\
 
 Motivated by the previous consideration,  we are led to study the  linear perturbation of the Yamabe problem
  \begin{equation}\label{Eq1}-\mathcal L_g u+\epsilon u=u^{N+2\over N-2},\ u>0,\ \hbox{in}\ (M,g) \end{equation}
 where the first eigenvalue of   $-\mathcal L_g $ is positive and $\epsilon$ is a small parameter.
In particular,  we address   the following questions.
\begin{itemize}
\item[(i)]
 Do there exist  solutions to \eqref{Eq1} which blow-up as $\epsilon\to0$?  
 \item[(ii)] Do there exist solutions to  \eqref{Eq1}   with non-isolated   blow-up  points, namely with   {\em clustering} blow-up points?
 \item[(iii)] Do there exist solutions to  \eqref{Eq1}   with non-isolated simple blow-up  points, namely with   {\em towering} blow-up points?
\end{itemize}

  Concerning question (i), 
 Druet in \cite{d1} proved that equation \eqref{Eq1} does not have any blowing-up solution when    $\epsilon<0$ and   $N=3,4,5$
   (except when the manifold is conformally equivalent to the round sphere). It is completely open the case when the dimension is $N\ge6.$
   The situation is completely different when 
   $\epsilon>0.$  Indeed, if  $N=3$  no blowing-up solutions exist  as proved by  Li-Zhu  \cite{lz}, while if $m\ge4$  blowing-up solutions do exist  as shown by Esposito, Pistoia and Vetois in \cite{epv}. In particular, if the dimension $N\ge6$ and the manifold is not locally conformally flat, Esposito, Pistoia and Vetois built solutions which blow-up at non-vanishing stable critical points $\xi_0$ of the Weyl's tensor, i.e. $|\Weyl_g(\xi_0)|_g\not=0.$
 In this paper, we show that the blowing-up point $\xi_0$ is not-isolated as soon as it is a  non-degenerate minimum point of the Weyl's tensor. This result 
   gives a positive answer to question (ii).  Finally,   a positive answer to question (iii) has been giving by Morabito, Pistoia and Vaira in a forthcoming paper  \cite{mpv}.\\
  
 Now, let us state the main result obtained in this paper.
  \begin{theorem}\label{cluster}
Let  $(M, g)$   be not locally conformally flat and   $N\geq 7$.   Let $\xi_0\in M$ be a non-degenerate minimum point of $\xi \to |{\rm Weyl}_g (\xi)|_g^2.$ Then, for any $k\in \mathbb N$,   there exist $\xi_\eps^j\in M$ for $j=1, \ldots, k$ and  $\eps_k>0$ such that for all $\eps\in (0, \eps_k)$ the problem \eqref{Eq1} has a solution $(u_\eps)_\eps$ with $k$ positive peaks at $\xi_\eps^j$ and $\xi_\eps^j \to \xi_0$ as $\eps\to 0$. 
\end{theorem}
  
Let us point out that Robert and V\'etois in \cite{rv}   built solutions having clustering blow-up points for a special class of perturbed Yamabe type equations 
which look like
 \begin{equation}\label{Eq12}-\mathcal L_g u+\epsilon H u=u^{N+2\over N-2},\ u>0,\ \hbox{in}\ (M,g). \end{equation}
 where  the potential $H$ is chosen
 with $k$ distinct strict local maxima concentrating at a  point $\xi_0$ with  $|\Weyl_g(\xi_0)|_g\not=0.$
 Indeed, these maxima points  generate  solutions with $k$ positive peaks collapsing to $\xi_0$ as $\epsilon$ goes to zero.
 Their result    is related to a suitable choice of the potential $H,$ but actually
 our result shows that the clustering phenomena  is intrinsic in the geometry of the manifold.
 \\
 
 Let us give an example. The warped product   $ \(S^{n} \times S^{m } , g_{S^{n}}\otimes f^2  g_{S^{m }}\),$  is the Riemannian manifold
$ S^{n} \times S^{m } $ equipped with the metric   $g = g_{S^{n}}\otimes f^2  g_{S^{m}}$.
Here $f :  S^{n} \to \R$ is a positive function called warping function.
 It is easy to see that if the warping function $f\equiv1$ than the the product manifold $ \(S^{n} \times S^{m } , g_{S^{n}}\otimes    g_{S^{m }}\) $  has Weyl tensor different from zero at any point. Using similar argoments to the ones used in \cite{pv}, we can prove that for generic   warping functions $f$ close to the constant $1$, the Weyl tensor has a non-degenerate 
 and non-vanishing minimum point.
 \\

The proof of our result  relies on a  finite  dimensional  Ljapunov-Schmidt reduction, whose main steps are described in Section 3
and their proofs are postponed in Section 4. Section 2 is devoted to recall some known results.

\section{Preliminaries}

We   provide the Sobolev space $H^1_g\(M\)$ with the scalar product
\begin{equation}\label{Eq6}
\<u,v\> =\int_M \left<\nabla u,\nabla v\right>_g d\nu_g+\beta_N\int_M\Scal_g uvd\nu_g
\end{equation}
where $d\nu_g$ is the volume element of the manifold. Here $\beta_N:={N-2\over 4(N-1)}.$ We let $\|\cdot\| $ be the norm induced by $\<\cdot,\cdot\> $. Moreover, for any function $u$ in $L^q(M)$, we denote the $L^q$-norm of $u$ by $\|u\|_q=\left(\int_M|u|^qd\nu_g\right)^{1/q}$.

We let $\i^*:L^{\frac{2N}{N+2}}\(M\)\rightarrow H^1_g(M)$ be the adjoint operator of the embedding $\i:H^1_g\(M\)\hookrightarrow L^{{2^*}}\(M\)$, i.e. for any $w$ in $L^{\frac{2N}{N+2}}(M)$, the function $u=\i^*\(w\)$ in $H^1_g\(M\)$ is the unique solution of the equation $ -\varDelta_gu+\beta_N\Scal_g u=w$ in $M$. By the continuity of the embedding of $H^1_g\(M\)$ into $L^{2^*}\(M\)$, we get
\begin{equation}\label{Eq7}
\left\|\i^*\(w\)\right\|\le C\left\|w\right\|_{\frac{2N}{N +2}}
\end{equation}
for some positive constant $C$ independent of $w$.
We rewrite problem \eqref{Eq1} as
\begin{equation}\label{Eq1b}
u=\i^*\[f(u) -\eps u\],\qquad u\in{\rm H}^1_g(M)\end{equation}
where we set
  $f(u):=(u^+)^p $ with  $p={N+2\over N-2} $.

We also define the energy $J_\epsilon:H_g^1(M)\to \mathbb R$
\begin{equation}\label{energy}
J_\epsilon(u):={1\over2}\int\limits_M\(|\nabla_g u|^2+\beta_N\Scal_g u^2+\epsilon u^2\)d\nu_g-{1\over p+1}\int\limits_M \(u^+\)^{p+1}d\nu_g,
\end{equation}
whose critical points are solutions to the problem \eqref{Eq1}.
\\
 
We are going to read the euclidean bubble defined in \eqref{bubble}
on the manifold via a geodesic normal coordinate system around a point $\xi\in M,$ i.e.
$$\mathcal U_{\mu,\xi}(z)= U_{\mu,0} \left( {\rm exp}_{\xi}^{-1}(z)    \right)=\mu^{-{N-2\over2}}U  \left(\frac{{\rm exp}_{\xi}^{-1}(z) }{\mu } \right), \ z\in B_g(\xi,r).$$
It is necessary to write the conformal laplacian in geodesic normal coordinates around the point $\xi.$ In particular, if $x\in B(0,r)$
using standard properties of the exponential map we can write
\begin{equation}\label{lb}
-\Delta_g u = -\Delta u -(g^{ij}-\delta^{ij})\partial^2_{ij}u+g^{ij}\Gamma^k_{ij}\partial_k u, 
\end{equation}
with
\begin{equation}\label{gij}
g^{ij}(x)=\delta^{ij}(x) -\frac 13 R_{iab j}(\xi)x_a x_b +O(|x|^3)\ \hbox{and}\ 
g^{ij}(x)\Gamma^k_{ij}(x)=\partial_l \Gamma^k_{ii}(\xi)x_l +O(|x|^2).
\end{equation}
Here $ R_{iab j}$ denotes the Riemann curvature tensor and $\Gamma^k_{ii}$ the Christoffel's symbols.
 Therefore, if we compare the conformal laplacian   with the euclidean laplacian of the bubble  the error at main order looks like
 $$ \mathcal L_g \mathcal U_{\mu,\xi}   -\Delta \mathcal U_{\mu,\xi} \sim  -\frac 13 \sum\limits_{a,b,i,j=1}^NR_{iab j}(\xi)x_a x_b\partial^2_{ij}U_{\mu,0}+\sum\limits_{i, l,k=1}^N \partial_l \Gamma^k_{ii}(\xi)x_l\partial_k U_{\mu,0}+{\beta_N R_g(\xi)}U_{\mu,0}.$$
For later purposes, it is necessary 
  to kill this main term    by adding to the bubble an higher order term $V$ which is defined as follows.
First, we remind 
  that any  solution of the linear equation (see \cite{be})
\begin{equation}\label{Eqlin}
-\Delta v=pU^{p-1}v\quad\hbox{in}\ \mathbb{R}^N,\end{equation}
is a linear combination of the functions
\begin{equation}\label{Eq14}
\psi^0\(x\)=x\cdot\nabla U(x)+{N-2\over2} U(x)
\ \hbox{and}\
 \psi^i\(x\)=\partial_i U(x), \ i=1,\dots,N.
\end{equation}
Next, we introduce the higher order term $V$ which has been built in Section 2 in \cite{ep}.

 \begin{proposition}\label{bubbleV}
 For any point $\xi\in M,$
 there exist  $\nu(\xi)\in\mathbb R $ and a function $V\in   \mathcal{D}^{1,2}(\mathbb{R}^N) $ solution to
\begin{align}\label{EqV1}
  -\Delta  V-f'(U) V = & -\sum\limits_{a,b,i,j=1}^N{1\over 3} R_{iabj} (\xi)x_ax_b\partial^2_{ij} U
 +\sum\limits_{i, l,k=1}^N \partial_l\Gamma^k_{ii}(\xi)x_l\partial_kU +\beta_N\Scal_g(\xi) U+\nu(\xi)\psi^0 \ \hbox{in}\ \mathbb{R}^N ,
\end{align}
with
$$\displaystyle\int\limits_{\mathbb{R}^N}V(x)\psi^i(x)dx=0,\  i=0,1,\dots,N.$$

Moreover, there exists $C\in\mathbb R$ such that
\begin{equation}\label{V3}
|V(x)|+ |x|\left|\partial_kV(x)\right|+|x|^2 \left|\partial^2_{ij}V (x)\right| \le C{1\over \(1+|x|^2\)^{N-4\over2}},\quad x\in\mathbb{R}^N. \end{equation}
 \end{proposition}

\section{Clustering}

\subsection{The ansatz: the cluster}
Let $r_0$ be a positive real number less than the injectivity radius of $M$  and $\chi$ be a smooth cut-off function  such that $0\le\chi\le1$ in $\mathbb{R}$, $\chi\equiv1$ in $ \[-r_0/2,r_0/2 \]$, and $\chi\equiv0$ out of $ \[-r_0,r_0 \]$. Let also  $\eta$ be a smooth cutoff function  such that $0\le\eta\le1$ in $\mathbb{R}$, $\eta\equiv1$ in $ \[-1,1 \]$, and $\eta\equiv0$ out of $ \[-2,2 \]$.

Let $k\ge1$ be a fixed integer.
Assume that $   \xi_0 \in M $ is a non degenerate minimum point   of   $\xi\to \left|\Weyl_g\(\xi \)\right|^2_g$ with  $\left|\Weyl_g\(\xi_0\)\right|\not=0$, i.e. 
\begin{equation}\label{xio}
 \nabla_g  \left|\Weyl_g\(\xi_0 \)\right|^2_g=0 \ \hbox{and   the quadratic form $\mathcal Q(\xi_0):=D_g^2 \left|\Weyl_g\(\xi_0 \)\right|^2_g $ is positive definite.}
\end{equation} 
 Set
 \begin{equation}\label{do}
 d_0 := \( { B_N\over 2 A_N\left|\Weyl_g\(\xi_0\)\right|^2_g}\)^{1/2}\quad \hbox{($A_N$  and $B_N$ are positive constants defined in   \eqref{costanti})}
\end{equation} 
and let us choose
\begin{equation}\label{taui}
	 \tau_1,\dots,\tau_k\in\mathbb{R}^N\ \hbox{with}\ \tau_i\not=\tau_j\ \hbox{if}\ i\not=j
\end{equation}
and for any $i=1,\dots,k$
 \begin{equation} \label{mui}
	  \mu_i=\eps^\alpha\(d_0+d_i\eps^\beta\),\ \hbox{where}\ d_1,\dots,d_k\in(0,+\infty),\ 
	\alpha:={1\over2},\ \beta:={N-6\over2N}.
\end{equation}
Then, let us define
\begin{align}\label{Wj}
{\mathcal W_i}(z):= & \chi(d_g(z, \xi_0))\mu_i^{-\frac{N-2}{2}}U\left(\frac{{\rm exp}_{\xi_0}^{-1}(z)-\eps^\beta\tau_i}{\mu_i} \right)   \nonumber\\ &+\mu_i^2  \eta\({\left|\exp_{\xi_0}^{-1}\(z\)-\eps^\beta\tau_i\right|\over\mu_i}\) \chi(d_g(z, \xi_0))\mu_i^{-\frac{N-2}{2}}V\left(\frac{{\rm exp}_{\xi_0}^{-1}(z)-\eps^\beta\tau_i}{\mu_i} \right) ,\ z\in M
\end{align}
where the functions $U$ and $V$ are defined, respectively, in \eqref{bubble} and \eqref{EqV1}. 
 Set
$$ \mathcal C :=\{( \tau_1,\dots,\tau_k)\in\mathbb{R}^{kN}\ :\ \tau_i\not=\tau_j\ \hbox{if}\ i\not=j\}.$$

We look for solutions of equation \eqref{Eq1}  or \eqref{Eq1b} of the form
\begin{equation}\label{uk}
	u_\eps  (z)=\sum\limits_{i=1}^k\mathcal{W}_{i} (z )   +\phi_{\eps   }(z),
\end{equation}
where  the remainder term  $\phi_{\eps  }$ belongs to the space $ \mathcal K^\perp$ defined as follows.
 For any  $i=1,\dots,k$  we introduce the functions
\begin{equation}\label{Eq16}
	Z_ {j,i}\(z\)=\chi \(d_g (z,\xi_0 ) \)\mu_i^{-\frac{N-2}{2}}\psi^j\left(\frac{{\rm exp}_{\xi_0}^{-1}(z)-\eps^\beta\tau_i}{\mu_i} \right),\ j=0,1,\dots,N,
\end{equation}
where the functions $\psi^j  $ are defined in \eqref{Eq14}. We define the subspaces
$$\mathcal  K :=\Span\left\{\i^*\(Z_{j,i}\), j=0,1,\dots,N,\ i=1,\dots,k\right\}$$ and $$ \mathcal K^\perp :=\left\{\phi\in H^1_g\(M\):\ \left\langle\phi,\i^*\(Z_{j,i} \)\right\rangle    =0,\quad j=0,\dotsc,N,\ i=1,\dots,k\right\} $$
and we also define the projections $\Pi $ and $\Pi^\perp $ of   $H^1_g\(M\)$ onto  $\mathcal  K$
and
$\mathcal  K^\perp$, respectively.

Therefore, equation \ref{Eq1b} turns out to be  equivalent to the system
\begin{align}\label{sk1}
	& \Pi^\perp  \{u_\eps  -\i^*\[f(u_\eps  )- \eps u_\eps  \]\}=0,\\
	& \Pi \{u_\eps  -\i^*\[f(u_\eps  )-\eps u_\eps  \]\}=0.
	\label{sk2}
\end{align}
where $u_\eps  $ is given in \eqref{uk}.
\subsection{The remainder term: solving the equation \eqref{sk1}}
In order to find the remainder term $\phi_{\eps  }$ we rewrite \eqref{sk1} as
$$\mathcal E +\mathcal L (\phi_{  \eps})+\mathcal N (\phi_ \eps)=0,$$
where the error term  $\mathcal E $ is  defined by
\begin{equation}\label{errcl}
\mathcal E:=\Pi^\bot \left\{\sum_{i=1}^k \mathcal W_i -\i^*\left[f\left(\sum_{i=1}^k \mathcal W_i  \right)-\eps \sum_{i=1}^k \mathcal W_i  \right]\right\}
\end{equation}
the linear operator  $\mathcal L$ is defined by
\begin{equation}\label{lincl}
\mathcal L  (\phi_{ \eps}):=\Pi^\bot \left\{\phi_\eps-\i^*\left[f'\left(\sum_{i=1}^k \mathcal W_i  \right)\phi_\eps-\eps \phi_\eps\right]\right\}
\end{equation}
and the higher order term  $\mathcal N $ is defined by
\begin{equation}\label{enne}
\mathcal N :=\Pi^\bot \left\{-\i^*\left[f\left(\sum_{i=1}^k \mathcal W_i  +\phi_\eps\right)-f\left(\sum_{i=1}^k \mathcal W_i \right)-f'\left(\sum_{i=1}^k \mathcal W_i  \right)\phi_{  \eps}\right]\right\}.
\end{equation}

In order to solve equation \eqref{sk1}, first of all we need to evaluate the $H^1_g(M)-$ norm of the error term  $\mathcal E  $.
This is done in the following lemma whose proof is postponed in Section \ref{appB}.\\

\begin{lemma}\label{errorecluster}
  For any compact subset  $A\subset  (0, +\infty)^k \times\mathcal C$ there exists a positive constant $C $ and $\eps_0>0$ such that 
for any $(d_1,\dots,d_k,\tau_1,\dots,\tau_k)\in A$ and for any $\eps\in(0,\eps_0)$ it holds
\begin{equation}\label{stimaerrcl}
 \|\mathcal E \|\leq C  \left\{\begin{aligned} &\eps^{\frac 54} \qquad &\mbox{if}\,\, N=7\\ &\eps^{3\over2}|\ln\eps|^{\frac 5 8} \qquad &\mbox{if}\,\, N=8\\ &\eps^{3\over 2} \qquad &\mbox{if}\,\, N\ge 9.\end{aligned}\right.\end{equation}
\end{lemma}

Next, we need to understand the invertibility of the linear operators $\mathcal L   $.
This is done in the following lemma whose proof can be carried out as in   \cite{rv2}.\\
\begin{lemma}\label{invcluster}
For  any compact subset $A \subset (0,+\infty)^k\times \mathcal C $ there exists a positive constant $C $  and $\eps_0>0$ such that   for any $(d_1,\dots,d_k,\tau_1,\dots,\tau_k)\in A$ and for any $\eps\in(0,\eps_0)$ it holds
\begin{equation}
\|\mathcal L  (\phi )\|\geq C  \|\phi \|\ \hbox{for any $\phi \in \mathcal K ^\bot$}.\end{equation}
\end{lemma}

   Finally, we are able to solve equation \eqref{sk1}. This is done in the following proposition,  whose proof is postponed in Section \ref{appB} and relies on  a   standard contraction mapping argument.\\

\begin{proposition}\label{esistenzacluster}
 For any compact subset $A \subset  (0,+\infty)^{k}\times\mathcal C $ there exists a positive constant $C $ and $\eps_0$ such that for $\eps\in(0,\eps_0)$ and for any $(d_1,\dots,d_k, \tau_1,\dots,\tau_k )\in A $    there exists a unique function $\phi_{ \eps}\in \mathcal K^\bot $  which solves  equation  \eqref{sk1} such that
\begin{equation} 
 \|\phi_\eps\|\leq C   \eps^{{3(N-2)\over 2N }+\zeta} \end{equation}
 for some $\zeta>0.$
Moreover,   the map $(d_1,\dots,d_\ell, \tau_1,\dots,\tau_k )\to \phi_{\ell, \eps}(d_1,\dots,d_\ell,\tau_1,\dots,\tau_k ) $ is of class $C^1$ and 
$$\| \nabla _{(d_1,\dots,d_\ell,\tau_1,\dots,\tau_k ) } \phi_{ \eps}\|\leq C \eps^{{3(N-2)\over 2N }+\zeta} $$
for some   positive constants $C$ and $\zeta.$
 \end{proposition}

\subsection{The reduced problem: proof of Theorem \ref{cluster}}
Let us introduce the reduced energy, defined by
\begin{equation}\label{red-en-cl}
\widetilde{J}_\eps(d_1, \ldots, d_k,\tau_1, \ldots, \tau_k):=J_\eps\(\sum\limits_{i=1}^k\mathcal W_{i}+\phi_{ \eps}\),\ (d_1, \ldots, d_k,\tau_1, \ldots, \tau_k) \in (0, +\infty)^k \times (\mathbb R^N)^k
\end{equation} where the remainder term   $\phi_{  \eps}$ is defined in Proposition \ref{esistenzacluster}.

The following result allows as usual to reduce our problem to a finite dimensional one. The proof is standard and it is postponed in Section \ref{appB}.\\

\begin{proposition}\label{clu-rido}
\begin{itemize}
\item[(i)] $\sum\limits_{i=1}^k\mathcal W_{i}+\phi_{ \eps}$ is a solution to \eqref{Eq1} if and only if $(d_1, \ldots, d_k,\tau_1, \ldots, \tau_k ) \in (0, +\infty)^k \times (\mathbb R^N)^k$  is a critical point of the reduced energy \eqref{red-en-cl}
\item[(ii)] The following expansion holds true
\begin{equation}\begin{aligned}
& \widetilde{J}_\eps(d_1, \ldots, d_k,\tau_1, \ldots, \tau_k)  :=k D_N+c(\xi_0)\eps^2+\eps^{3{N-2\over N}}
\mathfrak J (d_1, \ldots, d_k,\tau_1, \ldots, \tau_k)+o\(\eps^{3{N-2\over N}}\)
\end{aligned}\end{equation}
as $\eps\to 0$, $C^0-$ uniformly with respect to $(d_1, \ldots, d_k,\tau_1, \ldots, \tau_k)$ in compact subsets of $(0, +\infty)^k\times \mathcal C.$
Here $c(\xi_0):=k\left[-A_N |{\rm Weyl}_g(\xi_0)|^2_g d_0^4+B_N d_0^2\right],$
 $A_N,$ $B_N$, $D_N$ and $E_N $ are positive constants defined in \eqref{costanti} and
\begin{equation}\label{fraj} \mathfrak J (d_1, \ldots, d_k,\tau_1, \ldots, \tau_k):=-{1\over2}A_N d_0^4\sum\limits_{i =1}^k\mathcal Q(\xi_0)(\tau_i,\tau_i) -E_Nd_0^{N-2}\sum\limits_{i,j=1\atop i\not=j}^k{1\over |\tau_i-\tau_j|^{N-2}}-B_N \sum\limits_{i=1}^kd_i^2 .\end{equation}

\end{itemize}
\end{proposition}

\begin{proof}[Proof of Theorem \ref{cluster}]
By (i) of Proposition \eqref{clu-rido},  it is sufficient to find a critical point of the reduced energy $\widetilde{J}_\eps$.  
   Now, the function $\mathfrak J$  defined in \eqref{fraj}, 
 has  a maximum point    which is stable under $C^0-$perturbations. Therefore,  by (ii) of Proposition \eqref{clu-rido}, we deduce that if $\eps$ is small enough there exists 
 $({d_1}_\eps, \ldots, {d_k}_\eps,{\tau_1}_\eps, \ldots, {\tau_k}_\eps)$ critical point of $\widetilde{J}_\eps.$ 
	That concludes the proof.

\end{proof}

\section {Appendix}\label{appB}
For any $i=1,\dots, k,$ we set
$${  W_i}(x):=    \mu_i^{-\frac{N-2}{2}}U\left(\frac{x-\eps^\beta\tau_i}{\mu_i} \right)   +  \eta\({\left|x-\eps^\beta\tau_i\right|\over\mu_i}\) \chi(d_g(z, \xi_0))\mu_i^{-\frac{N-6}{2}}V\left(\frac{x-\eps^\beta\tau_i}{\mu_i} \right) ,\ x\in \mathbb{R}^N.
$$
It is important to point out that there exists $c>0$ such that
\begin{equation}\label{wi1}
	\left|W_i(x)\right|\le c{\mu_i^{N-2\over2}\over |x -\eps^\beta\tau_i|^{N-2}}\ \forall\ x\in\mathbb{R}^N.
\end{equation}

\subsection{Proof of Lemma  \ref{errorecluster}}
 
It is easy to see that,  ($\nu(\xi) $ is defined in \eqref{EqV1})
\begin{equation*}\begin{aligned}\|\mathcal E\|&\le c\sum_{i=1}^k \left| -\Delta_g \mathcal W_i +(\beta_N \Scal_g+\eps) \mathcal W_i -\nu(\xi) Z_{0, i}-f(\mathcal W_i)\right|_{\frac{2N}{N+2}}+c\left|f\left(\sum_{i=1}^k \mathcal W_i\right)-\sum_{i=1}^k f(\mathcal W_i)\right|_{\frac{2N}{N+2}}\end{aligned}\end{equation*}
Arguing exactly as    in Lemma 3.1 of \cite{ep}, we can estimate each term
$$\left| -\Delta_g \mathcal W_i +(\beta_N \Scal_g+\eps) \mathcal W_i -\nu(\xi) Z_{0, i}-f(\mathcal W_i)\right|_{\frac{2N}{N+2}}=\left\{\begin{aligned}&O\(\eps^{\frac54}\)\ \hbox{if}\ N=7,\\ &O\(\eps^{\frac32}|\ln\eps|^{\frac58}\)\ \hbox{if}\ N=8,\\ &O\(\eps^{\frac32}\)\ \hbox{if}\ N\ge9.\\\end{aligned}\right.$$

 Next, we show that
 $$\left|f\left(\sum_{i=1}^k \mathcal W_i\right)-\sum_{i=1}^k f(\mathcal W_i)\right|_{\frac{2N}{N+2}}=O\(\eps^{3}\) .$$

  Set for any $h=1, \ldots, k$ $B_h := B(\eps^\beta\tau_h, \eps^\beta \sigma/2)$ where $\sigma>0$ and small enough. For \eqref{taui} $B_h \subset B(0, r_0)$ and they are disjoint. We write
\begin{equation*}
\begin{aligned}
&\left|f\left(\sum_{i=1}^k \mathcal W_i\right)-\sum_{i=1}^k f(\mathcal W_i)\right|_{\frac{2N}{N+2}}\le c \left[\int_{ B(0, r_0)}(1-\chi^{p+1}(|x|))|\cdots|^{\frac{2N}{N+2}}|g(x)|^{\frac 12}\, dx \right]^{\frac{N+2}{2N}}\\
&+c \left[\int_{B(0, r_0)\setminus \cup_h B_h}|\cdots|^{\frac{2N}{N+2}}|g(x)|^{\frac 12}\, dx \right]^{\frac{N+2}{2N}} +c\sum_{h=1}^k \left[\int_{B_h}|\cdots|^{\frac{2N}{N+2}}|g(x)|^{\frac 12}\, dx \right]^{\frac{N+2}{2N}}\\
&\le c  \sum_{i=1}^k\left[\int_{ B(0, r_0)}(1-\chi^{p+1}(|x|))|W_i|^{\frac{2N}{N-2}}|g(x)|^{\frac 12}\, dx \right]^{\frac{N+2}{2N}}+c \left[\int_{B(0, r_0)\setminus \cup_h B_h}|W_i|^{\frac{2N}{N-2}}|g(x)|^{\frac 12}\, dx \right]^{\frac{N+2}{2N}}\\
&+ c \sum_{h=1}^k \left[\int_{B_h}\left| W_h^{p-1}\sum_{i\neq h} W_i\right|^{\frac{2N}{N+2}}|g(x)|^{\frac 12}\, dx \right]^{\frac{N+2}{2N}}+c \sum_{h=1}^k \left[\int_{B_h}\left| \sum_{i\neq h} W_i\right|^{\frac{2N}{N-2}}|g(x)|^{\frac 12}\, dx \right]^{\frac{N+2}{2N}}.
\end{aligned}
\end{equation*}
Let us estimate each term in the previous expression. We use \eqref{wi1}.
\begin{equation*}
\begin{aligned}
 \sum_{i=1}^k\left[\int_{B(0, r_0)}(1-\chi^{p+1}(|x|))|W_i|^{\frac{2N}{N-2}}|g(x)|^{\frac 12}\, dx \right]^{\frac{N+2}{2N}}&\le c \sum_{i=1}^k \left[\int_{\mathbb R^N\setminus B(0, r_0)}\frac{\mu_i^N}{|x-\eps^\beta \tau_i|^{2N}}\, dx\right]^{\frac{N+2}{2N}}\\
&\le c \sum_{i=1}^k \frac{\mu_i^{\frac{N+2}{2}}}{\eps^{\beta \frac{N+2}{2}}}\left[\int_{\mathbb R^N\setminus B(0, r_0/\eps^\beta)}\frac{1}{|y-\tau_i|^{2N}}\, dy\right]^{\frac{N+2}{2N}}\\
&\le c \eps^{(\alpha-\beta)\frac{N+2}{2}+\alpha\frac{N+2}{2}}\le c\eps^{3\frac {N+2}{2N}},
\end{aligned}
\end{equation*}
\begin{equation*}
\begin{aligned}
\left[\int_{B(0, r_0)\setminus \cup_h B_h}|W_i|^{\frac{2N}{N+2}}|g(x)|^{\frac 12}\, dx \right]^{\frac{N+2}{2N}}&\le c \frac{\mu_i^{\frac{N+2}{2}}}{\eps^{\beta\frac{N+2}{2}}}\left[\int_{B(0, r_0/\eps^{\beta})\setminus \cup_h B (\tau_h, \sigma/2)}\frac{1}{|y-\tau_i|^{2N}}\,dy \right]^{\frac{N+2}{2N}}\\
&\le C \eps^{(\alpha-\beta)\frac{N+2}{2}}\le C \eps^{3\frac {N+2}{2N}},
\end{aligned}
\end{equation*}
\begin{equation*}
\begin{aligned}
 \sum_{h=1}^k \left[\int_{B_h}\left| W_h^{p-1}\sum_{i\neq h} W_i\right|^{\frac{2N}{N+2}}|g(x)|^{\frac 12}\, dx \right]^{\frac{N+2}{2N}}&\le c \sum_{h=1}^k\sum_{i\neq h} \left[\int_{B_h}\frac{\mu_h^{\frac{4N}{N+2}}}{|x-\eps^\beta\tau_h|^{\frac{8N}{N+2}}}\frac{\mu_i^{\frac{N(N-2)}{N+2}}}{|x-\eps^\beta\tau_i|^{\frac{2N(N-2)}{N+2}}}\, dx\right]^{\frac{N+2}{2N}}\\
&\le c \sum_{h=1}^k\sum_{i\neq h}\mu_h^2 \mu_i^{\frac{N-2}{2}}\left[\int_{B_h}\frac{1}{|x-\eps^\beta \tau_h|^{\frac{8N}{N+2}}}\, dx\right]^{\frac{N+2}{2N}}\\
&\le c\sum_{h=1}^k\sum_{i\neq h} \mu_h^2 \mu_i^{\frac{N-2}{2}}\left[\int_{B(0, \eps^\beta \sigma/2)}\frac{1}{|y|^{\frac{8N}{N+2}}}\, dy \right]^{\frac{N+2}{2N}}\\
&\le c \sum_{h=1}^k\sum_{i\neq h}\mu_h^2 \mu_i^{\frac{N-2}{2}}\eps^{\beta \frac{N-6}{2}}\le c\eps^{3\frac {N+2}{2N}}
\end{aligned}
\end{equation*}
and
\begin{equation*}
\begin{aligned}
\sum_{h=1}^k \left[\int_{B_h}\left| \sum_{i\neq h} W_i\right|^{\frac{2N}{N-2}}|g(x)|^{\frac 12}\, dx \right]^{\frac{N+2}{2N}}&\le c\sum_{h=1}^k\sum_{i\neq h} \left[\int_{B_h}\frac{\mu_i^N}{|x-\eps^\beta \tau_i|^{2N}}\, dx\right]^{\frac{N+2}{2N}}\\
&\le C \frac{\mu_i^{\frac{N+2}{2}}}{\eps^{\beta\frac{N+2}{2}}}\left[\int_{B(\tau_h, \sigma/2)}\frac{1}{|y-\tau_i|^{2N}}\, dy\right]^{\frac{N+2}{2N}}\le c \eps^{3\frac {N+2}{2N}}.
\end{aligned}
\end{equation*}

\subsection{Proof of Proposition \ref{clu-rido}}
It is quite standard to prove that 
$$J_\eps\(\sum_{i=1}^k \mathcal W_i+\phi_\eps\)= J_\eps\(\sum_{i=1}^k \mathcal W_i\)+\Theta $$
$C^0-$ uniformly with respect to $(d_1,\dots,d_k, \tau_1,\dots,\tau_k )$ in compact subset of $  (0,+\infty) ^k\times\mathcal C,$ where $\Theta $ is a   smooth function  such that 
$  |\Theta|, |\nabla \Theta |=O(\eps^{3{N-2\over  N}+\zeta} ) $ for some small $\zeta>0.$ 
We shall prove that  
	\begin{align}\label{enk1}
		&J_\eps  \(\sum\limits_{i=1}^k\mathcal{W}_i\)  =kD_N
		+k\eps^2\[-A_N\left|\Weyl_g\(\xi_0\)\right|^2_g +  B_N d_0^2 \] \nonumber\\ &+ \eps^{3{N-2\over N}}\left[-{1\over2}A_N d_0^4\sum\limits_{i =1}^k\mathcal Q(\xi_0)(\tau_i,\tau_i) -E_Nd_0^{N-2}\sum\limits_{i,j=1\atop i\not=j}^k{1\over |\tau_i-\tau_j|^{N-2}}-B_N \sum\limits_{i=1}^kd_i^2   \right]
  +\Theta,
	\end{align}
	where
	\begin{align}\label{costanti}
		 A_N:=\frac{  K_N^{-N}}{24N\(N-4\)\(N-6\)},\ 
		B_N:= \frac{2\(N-1\)K_N^{-N}}{N\(N-2\)\(N-4\)}, \ D_N:=\frac{K_N^{-N}}{N},\  E_N:= \alpha_N  \int\limits_{\mathbb{R}^N}U^p(y)dy
	\end{align}
	and $K_N$ is the best constant for the embedding of $D^{1,2}\(\mathbb{R}^N\)$ into $L^{2^*}\(\mathbb{R}^N\).$ 
	Here  $\Theta $ is a   smooth function  such that 
$  |\Theta|, |\nabla \Theta |=O(\eps^{3{N-2\over  N}+\zeta} ) $ for some small $\zeta>0.$

Let us prove \eqref{enk1}.
 
	\begin{align}\label{ek1}
		 J_\eps  \(\sum\limits_{i=1}^k\mathcal{W}_i\)  =&\underbrace{\sum\limits_{i=1}^kJ_\eps  \(\mathcal{W}_i\)}_I-
		\underbrace{\sum\limits_{j<i}\int\limits_M  f\(\mathcal{W}_i\)\mathcal{W}_j d\nu_g}_{II}
			\nonumber\\
		&+ \sum\limits_{i<j}\int\limits_M \[\nabla_g\mathcal{W}_i\nabla_g\mathcal{W}_j +\beta_N\Scal_g\mathcal{W}_i\mathcal{W}_j -f\(\mathcal{W}_i\)\mathcal{W}_j\]d\nu_g\nonumber\\
		&
		-\int\limits_M \[F \(\sum\limits_{i=1}^k\mathcal{W}_i\) - \sum\limits_{i=1}^kF\(\mathcal{W}_i\) -\sum\limits_{i\not=j}f\(\mathcal{W}_i\)\mathcal{W}_j\]d\nu_g +\eps \sum\limits_{i<j}\int\limits_M \mathcal{W}_i\mathcal{W}_jd\nu_g.
	\end{align}
	First of all, we estimate the two leading terms $I$ and $II$ in \eqref{ek1}.\\
	The term $I$   is given by the contribution of each bubble. Indeed, 
	in Section 4 of \cite{ep} it was proved that for any $i=1,\dots,k$
	\begin{equation}\label{ek2}
		J_\eps \(\mathcal{W}_i\)
		=D_N -A_N\left|\Weyl_g\(\xi_i\)\right|^2_g\mu_i^4 +\eps B_N \mu_i^2 + \left\{O(\eps^{\frac52} )\ \hbox{if}\ N=7,\  O(\eps^3|\ln\eps|^3 )\ \hbox{if}\ N=8,\ O(\eps^3 )\ \hbox{if}\ N\ge9\right\}.\end{equation}
			Now, by the choice  of $d_0$ in \eqref{do} and the choice of $\mu_i,$ $\alpha$ and $\beta $ in \eqref{mui}, we get
\begin{align*}
	&\left|\Weyl_g(\xi_i)\right|^2=\left|\Weyl_g(\xi_0)\right|^2+{1\over2}\mathcal Q(\xi_0)[\tau_i,\tau_i]\eps^{2\beta}+O\(\eps^{3\beta}\),\\
	& \mu_i^4= \eps^{4\alpha}\[d_0^4+4d_0^3d_i\eps^\beta+6d_0^2d_i^2\eps^{2\beta}+O\(\eps^{3\beta}\)\],\\
	& \mu_i^2= \eps^{2\alpha}\[d_0^2+2d_0d_i\eps^\beta+ d_i^2\eps^{2\beta} \].
\end{align*}
 Therefore, a straightforward computation shows that 
	\begin{align}\label{ek21}-A_N\left|\Weyl_g\(\xi_i\)\right|^2_g\mu_i^4 +\eps B_N \mu_i^2 &=
	\eps^2\[-A_N\left|\Weyl_g\(\xi_0\)\right|^2_g +  B_N d_0^2 \]\nonumber \\ & + \eps^{3{N-2\over N}}\left[-{1\over2}A_N d_0^4   \mathcal Q(\xi_0)(\tau_i,\tau_i)  -B_N \sum\limits_{i=1}^kd_i^2   \right]+O\(\eps^{\frac{7N-18}{2N}}\).
	 \end{align}
	 By \eqref{ek2} and \eqref{ek21} we deduce the estimate of $I.$\\
	 
	The term $II$ is given by the interaction of   different bubbles.
	For any $h=1,\dots,k$ let  $B_h:=B( \eps^\beta\tau_h,\eps^\beta\sigma/2)$. By \eqref{taui}  we deduce that   $B_h\subset B(0,r_0)$ provided $\sigma$ is small enough and they are disjoint.  
Therefore, if  $i\not=j$ 
	\begin{align}\label{ek210}
		  \int\limits_M  f\(\mathcal{W}_i\)\mathcal{W}_j d\nu_g&= \int\limits_{B_i}f\(W_i(x)\)W_j(x)\left|g(x)\right|^{1/2}dx
		 +\int\limits_{B(0,r_0)\setminus B_i}f\(W_i(x)\)W_j(x)\left|g(x)\right|^{1/2}dx\nonumber\\ &
		+\int\limits_{B( 0,r_0)}\[1-\chi ^{p+1}\(|x|\)\]
		f\(W_i(x)\)W_j(x) \left|g(x)\right|^{1/2}dx\nonumber\\
		&=E_Nd_0^{N-2}{1\over|\tau_i-\tau_j|^{N-2}}\eps^{3\frac{N-2}N}+O\(\eps^3\).
	\end{align}
	Indeed, the main term of \eqref{ek210} is given by
	\begin{align}\label{ek3}
		&\int\limits_{B_i}f\(W_i(x)\)W_j(x)\left|g(x)\right|^{1/2}dx\nonumber\\ &=\int\limits_{B_i}
		f\( \mu_i^{-{N-2\over2}}U \({x-\eps^\beta\tau_i\over\mu_i}\)+\mu_i^{-{N-6\over2}}\eta\({x-\eps^\beta\tau_i\over\mu_i}\)V \({x-\eps^\beta\tau_i\over\mu_i}\)\)\times \nonumber\\ &\times
		\( \mu_j^{-{N-2\over2}}U \({x-\eps^\beta\tau_j\over\mu_j}\)+\mu_i^{-{N-6\over2}}\eta\({x-\eps^\beta\tau_j\over\mu_j}\)V \({x-\eps^\beta\tau_j\over\mu_j}\)\)\left|g(x)\right|^{1/2}dx\
		\nonumber\\ & \hbox{($\eta\({x-\eps^\beta\tau_j\over\mu_j}\)=0$ if $x\in B_i$ and $\eps$ is small enough)}   \nonumber\\ &=\int\limits_{B_i}
		f\( \mu_i^{-{N-2\over2}}U \({x-\eps^\beta\tau_i\over\mu_i}\)+\mu_i^{-{N-6\over2}}\eta\({x-\eps^\beta\tau_i\over\mu_i}\)V \({x-\eps^\beta\tau_i\over\mu_i}\)\)\times \nonumber\\ &\times
		\( \mu_j^{-{N-2\over2}}U \({x-\eps^\beta\tau_j\over\mu_j}\)\) \left|g(x)\right|^{1/2}dx\
		\nonumber\\ &  \hbox{(setting $x-\eps^\beta\tau_i=\mu_i y$)}\nonumber\\ &=\mu_i^{N-2\over2}\int\limits_{B( 0,\eps^\beta\sigma/2\mu_i)}
		f\( U \(y\)+\mu_i^2\eta(|y|) V \(y\)\)
		{\alpha_N\mu_j^{N-2\over2}\over \(\mu_j^2+|\mu_j y+ \eps^\beta(\tau_i-\tau_j)|^2\)^{N-2\over2}}\times \nonumber\\ &\times
		\left|g(\mu_iy+\eps^\beta\tau_i)\right|^{1/2}dy\nonumber\\ &
		=\alpha_N{\mu_i^{N-2\over2}\mu_j^{N-2\over2}\over \eps^{\beta(N-2)}|\tau_i-\tau_j|^{N-2}}\[\int\limits_{\mathbb{R}^N}U^p(y)dy+O\(\mu_i^2\)+O\(\({\mu_i\over\eps^\beta}\)^N\)\]\nonumber\\ &
		=\eps^{3{N-2\over N}}{d_0^{N-2 }\over  |\tau_i-\tau_j|^{N-2}} \alpha_N\int\limits_{\mathbb{R}^N}U^p(y)dy +O\(\eps^3\),
	\end{align}
	because of the choice of $\mu_i$ in \eqref{mui}.
	Moreover, by \eqref{wi1}
	\begin{align*} 
		&\left|\int\limits_{B(0,r_0)\setminus B_i}f\(W_i(x)\)W_j(x)\left|g(x)\right|^{1/2}dx\right| \nonumber\\ & \ \le c
		\int\limits_{ \mathbb{R}^N \setminus B_i}   {\mu_i^{N+2\over2}\over |x-\eps^\beta\tau_i | ^{ N+2}}{\mu_j^{N-2\over2}\over |x-\eps^\beta\tau_j | ^{ N-2}} dx \  \hbox{(setting $x  =\eps^\beta y$)}
		\nonumber\\ &\le c{\mu_i^{N+2\over2}\mu_j^{N-2\over2}\over\beta^N}
		\int\limits_{ \mathbb{R}^N \setminus B(\tau_i,\sigma/2)}   {1\over |y- \tau_i | ^{N+2}}{\mu_j^{N-2\over2}\over |y-\tau_j | ^{ N-2}} dy
		\nonumber\\ & =O\({\mu_i^{N+2\over2}\mu_j^{N-2\over2}\over\eps^{\beta N}}\)=O\(\eps^3\)
	\end{align*}
	and  
	\begin{align*} 
		&\left|\int\limits_{B( 0,r_0)}\[1-\chi ^{p+1}\(|x|\)\]
		f\(W_i(x)\)W_j(x) \left|g(x)\right|^{1/2}dx\right|  =O\(\eps^{N\over2}\).
	\end{align*}

	Finally, let us prove that all the other terms in \eqref{ek1} are of higher order.\\

	By \eqref{lb} and \eqref{gij}, we deduce that
	$$ \left| \Delta_g\mathcal{W}_i +\beta_N\Scal_g\mathcal{W}_i  -f\(\mathcal{W}_i\) \right|\(\hbox{exp}_{\xi_0}(x)\)\le c {\mu_i^{N-2\over2}\over\(\mu_i^2+|x-\eps^\beta\tau_i|^2\)^{N-2\over2}}$$
	and so by \eqref{wi1} if
	  $i\not=j$ we have
	\begin{align*}
		& \left|\int\limits_M \[\nabla_g\mathcal{W}_i\nabla_g\mathcal{W}_j +\beta_N\Scal_g\mathcal{W}_i\mathcal{W}_j -f\(\mathcal{W}_i\)\mathcal{W}_j\]d\nu_g\right| \nonumber\\ &
		= \left|\int\limits_M \[\Delta_g\mathcal{W}_i +\beta_N\Scal_g\mathcal{W}_i  -f\(\mathcal{W}_i\)\]\mathcal{W}_j d\nu_g\right|
		\nonumber\\ &\le c\int\limits_{B( 0,r_0) }   {\mu_i^{N-2\over 2}\over |x-\eps^\beta\tau_i| ^{ N-2}}{\mu_j^{N-2\over2}\over |x-\eps^\beta\tau_j| ^{N-2 }} dx
		\ \hbox{(setting $x=\eps^\beta y$)}\nonumber\\ &\le c
		{ \mu_i^{N-2\over2}\mu_j^{N-2\over2}\over \eps^{\beta N-4}} \int\limits_{ \mathbb{R}^N }   {1\over |y-\tau_i| ^{ N-2}}{1\over |y-\tau_j| ^{ N-2  }} dx=O(\eps^3) .
	\end{align*}
	Moreover, 
	if $i\not=j$
	\begin{align*}
		&\left|\int\limits_M\mathcal{W}_i\mathcal{W}_jd\nu_g\right|=\left|\int\limits_{B(0,r_0) }W_i(x)W_j(x)\left|g(x)\right|^{1/2}dx \right| \nonumber\\ & \ \le 
		c\int\limits_{B( 0,r_0) }   {\mu_i^{N-2\over 2}\over |x-\eps^\beta\tau_i| ^{ N-2}}{\mu_j^{N-2\over2}\over |x-\eps^\beta\tau_j| ^{N-2 }} dx
		\ \hbox{(setting $x=\eps^\beta y$)}\nonumber\\ &\le c
		{ \mu_i^{N-2\over2}\mu_j^{N-2\over2}\over \eps^{\beta N-4}} \int\limits_{ \mathbb{R}^N }   {1\over |y-\tau_i| ^{ N-2}}{1\over |y-\tau_j| ^{ N-2  }} dx=O(\eps^3) .
	\end{align*}
	Finally, we have
	\begin{align*}
		&\int\limits_M \[F \(\sum\limits_{i=1}^k\mathcal{W}_i\) - \sum\limits_{i=1}^kF\(\mathcal{W}_i\) -\sum\limits_{i\not=j}f\(\mathcal{W}_i\)\mathcal{W}_j\]d\nu_g\nonumber\\  & 
		=\sum\limits_{h=1}^k \int\limits_{B_h}\[F \(\sum\limits_{i=1}^kW_i\) - \sum\limits_{i=1}^kF\(W_i\) -\sum\limits_{i\not=j}f\(W_i\)W_j\]\left|g(x)\right |^{1/2}dx\\ &+\int\limits_{B( 0,r_0) \setminus \cup_h B_h}\[F \(\sum\limits_{i=1}^kW_i\) - \sum\limits_{i=1}^kF\(W_i\) -\sum\limits_{i\not=j}f\(W_i\)W_j\]\left|g(x)\right |^{1/2}dx\\
	&	+\int\limits_{B( 0,r_0)}\[1-\chi ^{p+1}\(|x|\)\]
		\[F \(\sum\limits_{i=1}^kW_i\) - \sum\limits_{i=1}^kF\(W_i\) -\sum\limits_{i\not=j}f\(W_i\)W_j\] \left|g(x)\right |^{1/2}dx.
	\end{align*}
	It is immediate that
	$$\int\limits_{B( 0,r_0)}\[1-\chi ^{p+1}\(|x|\)\]
		\[F \(\sum\limits_{i=1}^kW_i\) - \sum\limits_{i=1}^kF\(W_i\) -\sum\limits_{i\not=j}f\(W_i\)W_j\] \left|g(x)\right |^{1/2}dx=O\(\eps^{N\over2}\).$$
Moreover, outside the $k$ balls we get
	
	\begin{align*}
		&\int\limits_{B( 0,r_0) \setminus \cup_h B_h}\left|F \(\sum\limits_{i=1}^kW_i\) - \sum\limits_{i=1}^kF\(W_i\) -\sum\limits_{i\not=j}f\(W_i\)W_j\right|\left|g(x)\right |^{1/2}dx\nonumber\\ &
		\le c  \sum\limits_{i\not=j} \int\limits_{B( 0,r_0) \setminus \cup_h B_h}  \(|W_i|^{p-1}W_j^2+ |W_j|^{p-1}W_i^2\)dx=O\(\eps^3\),
	\end{align*}
	because   if $2<q=p+1<3$
$$
		\left|(a+b)^q-a^q-b^q-qa^{q-1}b-qab^{q-1}
		\right|\le c \(a^2b^{q-2}+a^{q-2}b^2\)\ \hbox{for any}\ a,b>0 $$
	and  if $j\not=i$  
	\begin{align*}
		&\int\limits_{B( 0,r_0) \setminus \cup_h B_h}   |W_i|^{p-1}W_j^2 dx\le
		c\int\limits_{B( 0,r_0) \setminus \cup_h B_h}   {\mu_i^{2}\over |x-\eps^\beta\tau_i| ^{ 4}}{\mu_j^{N-2}\over |x-\eps^\beta\tau_j| ^{2(N-2) }} dx
		\ \hbox{(setting $x=\eps^\beta y$)}\nonumber\\ &\le c
		{ \mu_i^{2}\mu_j^{N-2}\over \eps^{\beta N}} \int\limits_{ \mathbb{R}^N \setminus \cup_h B(\tau_h,\sigma/2)}   {1\over |y-\tau_i| ^{ 4}}{1\over |y-\tau_j| ^{2(N-2) }} dx.
	\end{align*}
	On each ball $B_h$ we also have
	\begin{align}\label{ek7}
		&\int\limits_{B_h}\left|F \(\sum\limits_{i=1}^kW_i\) - \sum\limits_{i=1}^kF\(W_i\) -\sum\limits_{i\not=j}f\(W_i\)W_j\right|\left|g(x)\right\|^{1/2}dx
		\nonumber\\ &\le
		\int\limits_{B_h}\left|F \(W_h+\sum\limits_{i\not=h}W_i\) -F(W_h)-\sum\limits_{j\not=h}f(W_h)W_j\right| dx\nonumber\\ &
		+\sum\limits_{i\not=h}\int\limits_{B_h}\left|F\(W_i\)\right|dx+\sum\limits_{i\not=h\atop j\not=i}\int\limits_{B_h}
		\left|f\(W_i\)W_j\right| dx\nonumber
		\\ &\le c\sum\limits_{i\not=h}\int\limits_{B_h} W_h^{p-1}W_i^2dx+c\sum\limits_{i\not=h}\int\limits_{B_h} W_i^{p+1} dx+
		c\sum\limits_{i\not=h\atop j\not=i}\int\limits_{B_h}
		 W_i^pW_j dx
		,\end{align}
		because   if $ q=p+1\ge 1$
$$
		\left|(a+b)^q-a^q -qa^{q-1}b 
		\right|\le c \( b^{q}+a^{q-2}b^2\)\ \hbox{for any}\ a,b>0. $$
Now we use  \eqref{wi1} and we get if $i\not=h$
\begin{align*}
		&\int\limits_{B_h} W_h^{p-1}W_i^2dx\le c\int\limits_{B_h}
		{\mu_h^{  2}\over |x-\eps^\beta\tau_h| ^{4}}{\mu_i^{N-2 }\over |x-\eps^\beta\tau_i| ^{2(N-2) }} dx\ \hbox{(setting $x=\eps^\beta y$)}\nonumber\\ &\le c{ \mu_h^2\mu_i^{N-2}\over \eps^{\beta N}}\int\limits_{B(\tau_h,\sigma/2)}
		{1\over |y- \tau_h| ^{4}}{1\over |y-\tau_i| ^{2(N-2) }} dy=O\(\eps^3\), 
	\end{align*}
	 if $j,i\not=h$   
	\begin{align*}
		&\int\limits_{B_h} |W_i|^{p}W_j dx\le  c \int\limits_{B_h}
		{\mu_i^{N+2\over 2}\over |x-\eps^\beta\tau_i| ^{N+2 }}{\mu_j^{N-2\over 2}\over |x-\eps^\beta\tau_j| ^{N-2 }} dx
		\nonumber \\ &\le
		c {\mu_i^{N+2\over2}\mu_j^{N-2\over2}\over\eps^{2 \beta N}} |B_h| \le c {\mu_i^{N+2\over2}\mu_j^{N-2\over2}\over\eps^{  \beta N}}=O\(\eps^3\),  \end{align*}
	 if  $i\not=h$
	\begin{align*}
		&\int\limits_{B_h}   |W_i|^{p}W_h dx \le c  \int\limits_{B_h}
		{\mu_i^{N+2\over 2}\over |x-\eps^\beta\tau_i| ^{N+2 }}{\mu_h^{N-2\over 2}\over |x-\eps^\beta\tau_h| ^{N-2 }} dx
		\nonumber \\ &\le
		c  {\mu_i^{N-2\over2}\over\eps^{ \beta (N+2)}}\int\limits_{B_h}
		{1\over |x-\eps^\beta\tau_h| ^{N-2 }} dx  \le c  {\mu_i^{N-2\over2}\mu_h^{N-2\over2}\over\eps^{ \beta N}} =O\(\eps^3\)   \end{align*}
	and  if  $i\not=h$
	$$\int\limits_{B_h}  W_i^{p+1}  dx\le c
		 \int\limits_{B_h}{\mu_i^N\over |x-\eps^\beta\tau_i| ^{2N}}dx \le c {\mu_i^N\over \eps^{2\beta N}}|B_h| \le    c
		 {\mu_i^N\over \eps^{\beta N }}=O\(\eps^3\).
$$
	 That concludes the proof.

\end{document}